\documentclass[reqno,12pt]{amsart}

\usepackage[T1]{fontenc}
\usepackage{amssymb,amsfonts,amsbsy,graphicx,setspace,color}
\usepackage{mathtools,hyperref}
\usepackage{todonotes}
\usepackage{IEEEtrantools}
\usepackage{thmtools}

\newenvironment{ieee*}[1]{\begin{IEEEeqnarray*}{#1}}{\end{IEEEeqnarray*}\ignorespacesafterend}

\makeatletter

\everymath{\displaystyle} 

\theoremstyle{plain}
\newtheorem{thm}{Theorem}[section]  
\newtheorem{lmm}[thm]{Lemma}  

\numberwithin{equation}{section} 
\numberwithin{figure}{section} 
\newtheorem{prp}[thm]{Proposition}

\theoremstyle{remark}
\newtheorem{rmr}[thm]{Remark}

\theoremstyle{definition}

\newtheorem{cnj}[thm]{Conjecture}

\newcommand{\T}{\mathbb T}
\newcommand{\R}{\mathbb R}

\newcommand{\Rdn}{(\R^d)^N}
\newcommand{\eqdef}{\vcentcolon=}

\newcommand{\st}{\ | \ }
\newcommand{\ie}{i.e.}
\newcommand{\eps}{\varepsilon}
    
\def\abs#1{|#1|}


\date{\today}

\author{Ugo Bindini}
\address{Department of Mathematics and Statistics \\ University of Jyväskylä \\ Finland}
\email{ugo.bindini@sns.it}

\author{Luigi De Pascale}
\address{Dipartimento di Matematica e Informatica, Universit\`a di Firenze \\ Viale Morgagni 67/a, 50134 Firenze \\ Italy}
\email{luigi.depascale@unifi.it}
\urladdr{http://web.math.unifi.it/users/depascal/}
\author{Anna Kausamo}
\address{Dipartimento di Matematica e Informatica, Universit\`a di Firenze \\ Viale Morgagni 67/a, 50134 Firenze \\ Italy}
\email{akausamo@gmail.com}

\title[On deterministic solutions for Coulomb cost]{On deterministic solutions for multi-marginal optimal transport with Coulomb cost}

\begin{document}

\begin{abstract}  
In this paper we study the three-marginal optimal mass transportation problem for the Coulomb cost on the plane $\R^2$. 
The key question is the optimality of the so-called Seidl map, first disproved by Colombo and Stra. We generalize the partial positive result obtained by Colombo and Stra and give a necessary and sufficient condition for the radial Coulomb cost to coincide with a much simpler cost that corresponds to the situation where all three particles are aligned. Moreover, we produce an infinite class of regular counterexamples to the optimality of this family of maps.
 \end{abstract}
 
\maketitle

\section{Introduction}

\subsection{Multi-Marginal Optimal Mass Transportation Problem for Coulomb cost}
We denote by $\mathcal{P}(\R^d)$ the set of all Borel probability measures on the space $\R^d$ where $d\ge 1$ is the dimension of the space. In this paper we are interested in the Multi-Marginal Optimal Mass Transportation (MOT) problem for the Coulomb cost $\tilde c\colon \Rdn \to\R\cup\{+\infty\}$, 
\[\tilde c(x_1,\ldots,x_N)=\sum_{1\le i<j\le N}\frac{1}{|x_i-x_j|}.\]
 We fix a marginal measure $\tilde\rho\in\mathcal{P}(\R^d)$ and seek for minimizing the quantity
\begin{equation}\label{eq:mk}
 \int_{\Rdn}\tilde c \, d\gamma(x_1,\ldots,x_N)
\end{equation} 
over all \textit{couplings} $\gamma\in \mathcal{P}(\Rdn)$ of the marginal measures $\tilde\rho$, that is, over the set
\[\Pi_N(\tilde\rho) \eqdef \{\gamma\in \mathcal{P}(\R^d)~|~(pr_i)_\sharp\gamma=\tilde\rho\text{ for all }i\in \{1,\ldots,N\}\}\,,\]
where $pr_i$ is the projection on the $i$-th coordinate: 
\[pr_i(x_1,\ldots,x_N)=x_i~~~\text{for all }(x_1,\ldots,x_N)\in \Rdn. \]
Finding the minimal $\gamma$ for the problem \eqref{eq:mk} is often called solving the \emph{Monge-Kantorovich} (MK) problem, in honor of the French mathematician Gaspard Monge (1746--1818) and the Russian mathematician Leonid Vitaliyevich Kantorovich (1912--1986), both of whom can be considered founders of the field of optimal mass transportation. The existence of minimizers for the problem \eqref{eq:mk} is easily proven (we are minimizing a linear functional on a compact set) and their structure is rather well-understood (see \cite{de2015optimal}).
A much more challenging problem is the one of finding --- or even proving the existence of --- a \emph{deterministic} optimal coupling. A deterministic optimal coupling is a solution $\gamma_{opt}$ for the problem \eqref{eq:mk} of the type
\begin{equation}\label{eq:mongetype}
\gamma_{opt}=(Id,T_1,\ldots, T_{N-1})_\sharp\tilde\rho,
\end{equation}
where $T_i\colon \R^d\to \R^d$ are Borel functions such that $(T_i)_\sharp\tilde\rho=\tilde\rho$ for all $i\in \{1,\ldots,N-1\}$. This is equivalent to asking whether the equality
\begin{align*}
\min_{\gamma\in \Pi_N(\tilde\rho)} & \int_{\Rdn} \tilde c(x_1,\ldots, x_N)\, d\gamma(x_1,\ldots, x_N)\\
& = \min\left\{\int_{\R^d}\tilde c(x,T_1(x),\ldots, T_{N-1}(x))\, d\tilde\rho(x)~\bigg|~(T_i)_\sharp\tilde\rho=\tilde\rho\right\}\end{align*}
holds. If the answer is affirmative, we call any minimizing coupling of the type \eqref{eq:mongetype} a \emph{Monge minimizer}.

Since the cost function $\tilde c$ is symmetric with respect to permuting the coordinates and the density $\tilde\rho$ has no atoms, in view of the conjecture we are studying and also \cite{colombo2015equality, ghoussoub2014remarks, ghoussoub2014symmetric} we may restrict ourselves to seeking for Monge minimizers of the type
\[\gamma_{opt}=(Id, T, \dotsc, T^{N-1})_\sharp\tilde\rho,\]
where $T\colon \R^d\to\R^d$ is a Borel measurable function such that $T_\sharp\tilde\rho=\tilde\rho$ and $T^N=Id$. Here and from now on we denote for all natural numbers $k$ by $T^k$ the $k$-fold composition of $T$ with itself.

This paper concerns the case where $N=3$, $d=2$ and the density $\tilde\rho$ is radially symmetric, that is, $A_\sharp\tilde\rho = \tilde\rho$ for all $A\in SO(2)$. In this case (and also for general $N$ and $d$ as long as $\tilde\rho$ is radially symmetric) solving the MOT problem for Coulomb cost can be reduced to a one-dimensional problem where the underlying space is the positive halfline $\R_+:=[0,\infty)$. To make this notion precise, we define the \emph{radial cost} $c\colon (\R_+)^3 \to \R\cup\{+\infty\}$,
\begin{equation*}\label{radial-cost}
	\begin{split}
	c(r_1,r_2,r_3) = \min\left\{ \tilde c(v_1,v_2,v_3)~|~|v_i|=r_i
	\text{ for }i=1,2,3\right\} \\ \text{for all }(r_1,r_2,r_3)\in (\R_+)^3.
	\end{split}
\end{equation*}

For a given triple $(r_1,r_2,r_3)$ there exist many differently-oriented vectors $(v_1,v_2,v_3)$ that realize the above minimum.  Once a triple of minimizers $(v_1,v_2,v_3)$ has been fixed, the optimal configuration can be characterized by giving the radii and the angles between them.
We may always assume that the vector $v_1$ lies along the positive $x$-axis. With this choice in mind we denote by $\theta_2$ the angle between $v_1$ and $v_2$ and by $\theta_3$ the angle between $v_1$ and $v_3$. For this radial and angular data that corresponds to the triple of vectors $(v_1,v_2,v_3)\in (\R^2)^3$ we will sometimes use the notation $C(r_1,r_2,r_3,\theta_2,\theta_3)$ for the Coulomb cost $\tilde c(v_1,v_2,v_2)$. This allows to rewrite the radial cost function $c$ as
\begin{equation} \label{eq:radial-cost}
	c(r_1,r_2,r_3) = \min_{(\theta_2, \theta_3) \in \T^2} C(r_1, r_2, r_3, \theta_2, \theta_3)
\end{equation}
We also introduce \emph{the radial density }$\rho=|\cdot|_\sharp\tilde\rho$.\footnote{Here $\abs{\cdot}$ denotes the function $\abs{\cdot} \colon \R^2 \to \R_+$ given by $\abs{(x,y)} = \sqrt{x^2+y^2}$.} Now solving the (MK) problem for the Coulomb cost and the marginal measure $\tilde\rho$ is equivalent to solving the one-dimensional (MK) problem in the class $\Pi_3(\rho)$ for the radial density $\rho$ and the radial cost $c$, as will be made more rigorous in the next theorem, first proven by Pass (see \cite[Section 3]{pass2013remarks}). 
\begin{thm}
The full (MK) problem for the Coulomb cost
\begin{equation}\label{eq:mkoriginal}
\min\left\{\int_{(\R^2)^3}\tilde c(v_1,v_2,v_3) \, d\tilde\gamma(v_1,v_2,v_3)~|~\tilde\gamma\in\Pi_3(\tilde\rho)\right\}
\end{equation}
and the corresponding radial problem
\begin{equation} \label{eq:radial} 
\min\left\{\int_{(\R_+)^3}c(r_1,r_2,r_3)\, d\gamma~|~\gamma\in\Pi_3(\rho)\right\}
\end{equation}
are equivalent in the following sense: the measure $\gamma\in\Pi_3(\rho)$ is optimal for the problem \eqref{eq:radial} if and only if the measure 
\[\tilde\gamma:=\gamma(r_1,r_2,r_3)\otimes\mu^{r_1,r_2,r_3}\]
is optimal for the problem \eqref{eq:mkoriginal}. Above, $\mu^{r_1,r_2,r_3}$ is the singular probability measure on the $3$-dimensional torus defined by
\[\mu^{r_1,r_2,r_3}=\frac{1}{2\pi}\int_0^{2\pi}\delta_{t}\delta_{\theta_2+t}\delta_{\theta_3+t}\, dt,\]
where $(\theta_2, \theta_3)$ are minimizing angles $\theta_2=\angle (v_1,v_2)$, $\theta_3=\angle(v_1,v_3)$ for
\[ c(r_1,r_2,r_3) = \min\left\{ \tilde c(v_1,v_2,v_3)~|~|v_i|=r_i \text{ for }i=1,2,3\right\}. \]
\end{thm}

In \cite{seidl2007strictly} the authors conjectured the solution to the radial problem \eqref{eq:radial}. The conjecture is stated for all $d$ and $N$ but for the sake of clarity we formulate it here for $N=3$. 
\begin{cnj}[Seidl]\label{conj:seidl}
Let $\tilde\rho\in\mathcal{P}(\R^d)$ be radially-symmetric with radial density $\rho$.  Let $s_1$ and $s_2$ be such that 
\[\rho([0,s_1))=\rho([s_1,s_2))=\rho([s_2,\infty))=\frac13.\]
We define the map $T:[0,\infty)$ to be the unique map that sends, in the way that preserves the density $\rho$, 
\begin{itemize}
\item
the interval $[0,s_1)$ to the interval $[s_1,s_2)$ decreasingly,
\item
the interval $[s_1,s_2)$ to the half-line $[s_2,\infty)$ decreasingly, and 
\item
the half-line $[s_2,\infty)$ to the interval $[0,s_1)$ increasingly. 
\end{itemize}
More formally, this map is defined as 
\[T(x)=\begin{cases}
F^{-1}\left(\tfrac23-F(x)\right)&\text{ when }x\in[0,s_1)\\
F^{-1}\left(\tfrac43-F(x)\right)&\text{ when } x\in[s_1,s_2)\\
F^{-1}\left(1-F(x)\right)&\text{ when }x\in[s_2,\infty),
\end{cases}\]
where $F$ is the cumulative distribution function of $\rho$, that is, $F(r)=\rho([0,r))$. 

Then the map $T$ is optimal for the radial problem \eqref{eq:radial}.
\end{cnj}
The map introduced in \autoref{conj:seidl} is also called ``The Seidl map'' or  ``the $DDI$ map'' where the letters $DDI$ stand for \emph{Decreasing, Decreasing, Increasing}, identifying the monotonicities in which the first interval is mapped on the second, the second on the third, and finally the third back on the first. In an analogous manner one can define maps with different monotonicities: $III$, $IID$, $DDI$ and so on.  Since the marginals of  our MOT problem are all the same and equal to $\rho$, the only maps $T$ that make sense satisfy $T^3=Id$, which leads us to the so-called $\mathcal{T}:=\{I,D\}^3$ class, first introduced by Colombo and Stra in \cite{colombo2016counterexamples}:
\[\mathcal{T}:=\{III,DDI,DID,IDD\}.\]

In \cite{colombo2016counterexamples} the authors were the first to disprove Seidl's conjecture. They showed that for $N=3$ and $d=2$ the $DDI$ map fails to be optimal if the marginal measure is concentrated on a very thin annulus. They also provided a positive example for the optimality of the $DDI$ map: they constructed a density, concentrated on a union of three disjointed intervals the last of which is very far from the first two, so that the support of the transport plan given by the $DDI$ map is $c$-cyclically monotone. On the other hand, in \cite{de2019c} De Pascale proved that also for the Coulomb cost the $c$-cyclical monotonicity implies optimality: this implication had been previously proven only for cost functions that can be bounded from above by a sum of $\rho$-integrable functions. Using these results and making the necessary passage between the radial problem \eqref{eq:radial} and the full problem \eqref{eq:mkoriginal} one gets the optimality of the $DDI$ map for the example of Colombo and Stra. 
In \cite{colombo2016counterexamples} the authors also provided a counterexample for the non-optimality of all transport maps in the class $\mathcal{T}$.

In this paper we address the connection between the density $\rho$ and the optimality or non-optimality of the Seidl map for $d=2$ and $N=3$. 
Our main results are the following:
\begin{thm} \label{monge-solution}
	Let $\rho \in \mathcal{P}(\R_+)$ such that 
	\begin{equation} \label{phi1-condition}
	r_2(r_3-r_1)^3 -r_1(r_3+r_2)^3 - r_3(r_1+r_2)^3 \geq 0
	\end{equation}
	for $\rho$-a.e. $(r_1,r_2,r_3) \in [0,s_1] \times [s_1,s_2] \times [s_2,s_3]$. Then the DDI map $T$ provides an optimal Monge solution $\gamma = (Id, T, T^2)_{\#}(\rho)$ to the problem \eqref{eq:radial}.
\end{thm}

This theorem makes more quantitative the positive result of Colombo and Stra (see Remarks \ref{functionfi} and \ref{ColomboStra} for a more detailed description). Its proof also gives a necessary and sufficient condition for the radial Coulomb cost to coincide with a much simpler cost that corresponds to the situation where all three particles are aligned. More precisely, we show that
\begin{thm}
	Let $0 < r_1 < r_2 < r_3$. Then $(\theta_2,\theta_3) = (\pi, 0)$ is optimal in \eqref{eq:radial-cost} if and only if
	\[ r_2(r_3-r_1)^3 -r_1(r_3+r_2)^3 - r_3(r_1+r_2)^3 \geq 0. \]
	Moreover, if \eqref{phi1-condition} holds, $(\theta_2, \theta_3) = (\pi, 0)$ is the unique minimum point.
\end{thm}

 We continue by using this new condition to construct a wide class of counterexamples for the optimality of the maps of the class $\mathcal{T}$. This class contains densities that are rather physical, for example positive, continuous and differentiable. 
\begin{thm} \label{thm-counterexample}
	Let $\rho \in \mathcal{P}(\R_+)$ positive everywhere such that $\frac{s_1}{s_2} > \frac{1 + 2\sqrt{3}}{5}$ and 
	\begin{equation}\label{eq:phi1onDDI}
	T(x)(T^2(x) - x)^3 - x(T^2(x) + T(x))^3 - T^2(x)(x + T(x))^3 \geq 0 
	\end{equation}
	for $\rho$-a.e. $x \in (0, s_1)$, where $T$ is the DDI map. Then \textit{none} of the maps $S$ in the class $\{D,I\}^3$ provides an optimal Monge solution $\gamma = (Id, S, S^2)_{\#}(\rho)$ to the problem \eqref{eq:radial}. Moreover, there exist smooth counterexample densities. 
\end{thm}

\section{A study of the $c$ cost}
Recall the definition of the radial cost function \eqref{radial-cost}
\[ c(r_1, r_2, r_3) = \min \left\{ \tilde{c} (v_1, v_2, v_3) \st \abs{v_i} = r_i \text{ for all } i=1,2,3 \right\}. 
\]
 
By rotational invariance, we can suppose that the minimum is always achieved for $v_1 = r_1 \hat{x}$, \ie, in polar coordinates, $\theta_1 = 0$. Hence we can say that
\begin{equation} \label{radial-cost-2}
c(r_1, r_2, r_3) = \min \left\{
C (r_1, 0; r_2, \alpha; r_3,\beta) \st (\alpha, \beta) \in \T^2\right\};
\end{equation}
due to the dimension $d=2$, the cost $C(r_1, 0; r_2, \alpha; r_3,\beta)$ has an explicit expression
\begin{eqnarray*}
	C(r_1, 0; r_2, \alpha; r_3,\beta) &=& \frac{1}{(r_1^2+r_2^2-2r_1r_2\cos \alpha)^{1/2}} + \frac{1}{(r_1^2+r_3^2-2r_1r_3\cos \beta)^{1/2}} \\
	& &+ \frac{1}{(r_2^2+r_3^2-2r_2r_3\cos (\alpha - \beta))^{1/2}}
\end{eqnarray*}

The main ingredient for the proof of \autoref{monge-solution} and \autoref{thm-counterexample} is the following result, already stated in the introduction, which we recall for the sake of reader.

\begin{thm} \label{thm-phi1}
	Let $0 < r_1 < r_2 < r_3$. Then $(\alpha, \beta) = (\pi, 0)$ is optimal for the definition of the radial cost $c$ if and only if the polynomial condition \eqref{phi1-condition}
	\[ 
	r_2(r_3-r_1)^3 -r_1(r_3+r_2)^3 - r_3(r_1+r_2)^3 \geq 0,
	\]
holds.	Moreover, if the above holds, $(\alpha, \beta) = (\pi, 0)$ is the unique minimum point.
\end{thm}

\subsection*{Proof of \autoref{thm-phi1}}

Let $0 < r_1 < r_2 < r_3$ be fixed. In order to lighten the notation, we will omit the dependence on the radii when possible. We will also introduce the following functions for $i,j \in \{1,2,3\}$ and $\theta \in \T^1$:
\[ D_{ij}(\theta) = r_i^2 + r_j^2 - 2r_ir_j\cos \theta, \quad F_{ij}(\theta) = \frac{1}{D_{ij}(\theta)^{1/2}}  \]
It will be useful to compute the derivatives of $F_{ij}$, so we do it now:
\begin{eqnarray*}
	F_{ij}'(\theta) &=& -\frac{r_ir_j \sin \theta}{D_{ij}^{3/2}} \\
	F_{ij}''(\theta) &=& - \frac{r_ir_j\cos\theta}{D_{ij}^{3/2}} + \frac{3}{2} \frac{2r_i^2r_j^2 \sin^2 \theta}{D_{ij}^{5/2}} \\
	&=& -\frac{r_ir_j(r_ir_j\cos^2 \theta + (r_i^2+r_j^2)\cos \theta - 3r_ir_j)}{D_{ij}(\theta)^{5/2}}
\end{eqnarray*}
In order to simplify the notation even more, we denote
\[ Q_{ij}(t) = r_ir_jt^2 + (r_i^2+r_j^2)t - 3r_ir_j, \quad t \in [-1,1], \]
so that
\[ F_{ij}''(\theta) = -\frac{r_ir_j Q_{ij}(\cos \theta)}{D_{ij}(\theta)^{5/2}}. \]
Observe that $Q_{ij}(-1) = -(r_i+r_j)^2$ and $\quad Q_{ij}(1) = (r_i - r_j)^2$, so that
\begin{equation} \label{F-values}
F_{ij}''(0) = -\frac{r_ir_j}{\abs{r_i-r_j}^3} \quad \text{and} \quad F_{ij}''(\pi) = \frac{r_ir_j}{(r_i+r_j)^3}
\end{equation}

First we prove that if $(\alpha, \beta) = (\pi, 0)$ is optimal in \eqref{radial-cost-2}, then \eqref{phi1-condition} holds. Recall that the function to minimize is
\[ f(\alpha, \beta) = F_{12}(\alpha) + F_{13}(\beta) + F_{23}(\alpha - \beta) \]
and notice that $f \in C^\infty(\T^2)$. Thus, if $(\pi, 0)$ is minimal, it must be a stationary point with positive-definite Hessian. Let us compute the gradient and the Hessian of $f$:
\[ \nabla f(\alpha, \beta) =( F_{12}'(\alpha) + F_{23}'(\alpha - \beta), F_{13}'(\beta) - F_{23}'(\alpha - \beta) , \]
\[ H f(\alpha, \beta) = \begin{pmatrix}
F_{12}''(\alpha) + F_{23}''(\alpha - \beta) & -F_{23}''(\alpha - \beta) \\
-F_{23}''(\alpha - \beta) & F_{13}''(\beta) + F_{23}''(\alpha - \beta)
\end{pmatrix}. \]

Using \eqref{F-values}, we have
\[ H f(\pi, 0) = \begin{pmatrix}
\frac{r_1r_2}{(r_1+r_2)^3} + \frac{r_2r_3}{(r_2+r_3)^3} & -\frac{r_2r_3}{(r_2+r_3)^3} \\
-\frac{r_2r_3}{(r_2+r_3)^3} & -\frac{r_1r_3}{(r_3-r_1)^3} + \frac{r_2r_3}{(r_2+r_3)^3}
\end{pmatrix} \]
and
\begin{eqnarray*}
	\det Hf(\pi, 0) &=& -\frac{r_1^2r_2r_3}{(r_1+r_2)^3(r_3-r_1)^3} - \frac{r_1r_2r_3^2}{(r_2+r_3)^3(r_3-r_1)^3} \\
	&&+ \frac{r_1r_2^2r_3}{(r_1+r_2)^3(r_2+r_3)^3} \\
	&=& \frac{r_1r_2r_3[r_2(r_3-r_1)^3 -r_1(r_3+r_2)^3 - r_3(r_1+r_2)^3]}{(r_1+r_2)^3(r_2+r_3)^3(r_3-r_1)^3}.
\end{eqnarray*}

The positivity of $\det Hf(\pi, 0)$ implies the condition \eqref{phi1-condition}, which proves the first part of \autoref{thm-phi1}.

Now we assume that \eqref{phi1-condition} holds, and we want to get that $(\pi, 0)$ is the unique minimum point. The first (and most challenging) step is given by the following

\begin{prp} \label{prop-stationary}
	Suppose that $0 < r_1 < r_2 < r_3$ satisfy \eqref{phi1-condition}. Then $(0, 0)$, $(0,\pi)$, $(\pi,0)$, $(\pi,\pi)$ are the only stationary points of $f(\alpha, \beta)$.
\end{prp}

The proof of \autoref{prop-stationary} is quite technical and long. For the sake of clarity we postpone it to the end of this section, in order to keep focusing on the main result.

Since $\{0, \pi\}^2$ are the only stationary points, the global minimum of $f$ must be between them. By direct comparison of the values $f(0,0)$, $f(0,\pi)$, $f(\pi, 0)$, $f(\pi, \pi)$ we will conclude that $(\pi, 0)$ is the unique minimum point.

We compute
\begin{eqnarray*}
	f(0,0) &=& \frac{1}{r_2-r_1} + \frac{1}{r_3-r_2} + \frac{1}{r_3-r_1} \\
	 f(0, \pi) &=& \frac{1}{r_2-r_1} + \frac{1}{r_3+r_2} + \frac{1}{r_3+r_1} \\
	f(\pi,0) &=& \frac{1}{r_2+r_1} + \frac{1}{r_3+r_2} + \frac{1}{r_3-r_1} \\ f(\pi, \pi) &=& \frac{1}{r_2+r_1} + \frac{1}{r_3-r_2} + \frac{1}{r_3+r_1},
\end{eqnarray*}
and observe that clearly $f(0,0) > f(0, \pi)$. To deduce the other inequalities we observe that the function
\[ h(x,y) = \frac{1}{x-y} - \frac{1}{x+y} \quad \text{for $0 < y < x$}. \]
is decreasing in $x$ and increasing in $y$, so $h(r_3,r_1) < h(r_2,r_1) \Rightarrow f(\pi, 0) < f(0, \pi)$ and $h(r_3,r_1) < h(r_3,r_2) \Rightarrow f(\pi, 0) < f(\pi, \pi)$, as wanted.

\begin{proof}[Proof of \autoref{prop-stationary}]
	A stationary point $(\alpha, \beta)$ must solve $\nabla f = 0$, \ie,
	\[ \begin{cases} \displaystyle -\frac{r_1r_2\sin \alpha}{D_{12}(\alpha)^{3/2}} - \frac{r_2r_3\sin (\alpha - \beta)}{D_{23}(\alpha - \beta)^{3/2}} = 0 \\ \displaystyle  -\frac{r_1r_3\sin \beta}{D_{13}(\beta)^{3/2}} + \frac{r_2r_3\sin (\alpha - \beta)}{D_{23}(\alpha - \beta)^{3/2}} = 0 \end{cases} \]
	which we rewrite as
	\begin{equation} \label{stationary-system}
	\begin{cases} \displaystyle \frac{r_1r_2\sin \alpha}{D_{12}(\alpha)^{3/2}} + \frac{r_1r_3\sin \beta}{D_{13}(\beta)^{3/2}} = 0 \\ \displaystyle  \frac{r_1r_3\sin \beta}{D_{13}(\beta)^{3/2}} - \frac{r_2r_3\sin (\alpha - \beta)}{D_{23}(\alpha - \beta)^{3/2}} = 0. \end{cases}
	\end{equation}
	Observe that the four points $(\alpha, \beta) \in \{0, \pi\}^2$ are always solutions for \eqref{stationary-system}. We will study this system in detail for $\beta \in [0,\pi]$. The conclusions can then be derived for $\beta \in [-\pi, 0]$ by making use of the change of variables $\tilde{\alpha} = -\alpha$, $\tilde{\beta} = - \beta$.
To proceed in the computations, we perform a finer study of the function
\[ g_{ij}(\theta) = -F_{ij}'(\theta) = \frac{r_ir_j\sin \theta}{D_{ij}(\theta)^{3/2}}\,,\]
so that the optimality conditions \eqref{stationary-system} will be rewritten in the form
\begin{equation} \label{gstationary-system}
\begin{cases} \displaystyle g_{12}(\alpha)=-g_{13}(\beta) \\ \displaystyle 
g_{13}(\beta) =g_{23}(\alpha-\beta) .
\end{cases}
\end{equation}
We now prove that for every $\beta$ in $[0,\pi]$ there exists at least one and at most two $\alpha$'s such that each of the two equations is satisfied.   
 
The derivative of $g_{ij}$ is
\[ g_{ij}'(\theta) = r_ir_j\frac{Q_{ij}(\cos\theta)}{D_{ij}(\theta)^{5/2}} \]
and it vanishes for
\[ Q_{ij}(\cos \theta_{ij}) = 0 \Rightarrow \cos \theta_{ij} = \frac{-r_i^2-r_j^2+\sqrt{r_i^4+14r_i^2r_j^2+r_j^4}}{2r_ir_j} \in (0,1). \]
By looking at the sign of the second degree polynomial $Q_{ij}$, we conclude that $g_{ij}(\theta)$ is increasing from 0 to its maximum on $[0, \theta_{ij}]$ and decreasing to 0 on $[\theta_{ij}, \pi]$ 

\begin{lmm} \label{lemma-gij}
	For every $\theta \in [0,\pi]$, $g_{13}(\theta) \leq g_{12}(\theta)$ and $g_{13}(\theta) \leq g_{23}(\theta)$. (See \autoref{plotg}.)
\end{lmm}

\begin{proof}
	We claim that $0 \leq g_{13}'(0) \leq g_{12}'(0)$ and $g_{13}'(\pi) \geq g_{12}'(\pi) \geq 0$. Indeed, using \eqref{F-values},
	\[ g_{13}'(0) = -F_{13}''(0) = \frac{r_1r_3}{(r_3-r_1)^3} \geq 0, \quad \text{and} \quad g_{12}'(0) = \frac{r_1r_2}{(r_2-r_1)^3}, \]
	thus
	\[ g_{13}'(0) \leq g_{12}'(0) \iff r_3(r_2 - r_1)^3 \leq r_2(r_3 - r_1)^3 \]
	which is weaker than \eqref{phi1-condition}.
	
	On the other hand,
	\[ g_{13}'(\pi) = -F_{13}''(\pi) = -\frac{r_1r_3}{(r_3+r_1)^3} \leq 0, \quad \text{and} \quad g_{12}'(\pi) = -\frac{r_1r_2}{(r_1+r_2)^3}, \]
	thus
	\[ g_{13}'(\pi) \geq g_{12}'(\pi) \iff r_3(r_1+r_2)^3 \leq r_2(r_3+r_1)^3 \]
	which is once again weaker than \eqref{phi1-condition}.
	
	Moreover, the equation $g_{13}(\theta) = g_{12}(\theta)$ has at most one solution in $(0, \pi)$, since we have the following chain of equivalent equalities:
	\begin{eqnarray*}
		g_{13}(\theta) &=& g_{12}(\theta) \\
		\frac{r_1r_3}{D_{13}(\theta)^{3/2}} &=& \frac{r_1r_2}{D_{12}(\theta)^{3/2}} \\
		r_3^{2/3}(r_1^2 + r_2^2 - 2r_2r_3\cos \theta) &=& r_2^{2/3}(r_1^2 + r_3^2 - 2r_1r_3\cos \theta)
	\end{eqnarray*}
	\[  \cos \theta = \frac{r_2^{2/3}r_3^{2/3}(r_3^{4/3}-r_2^{4/3}) - r_1^2(r_3^{2/3} - r_2^{2/3})}{2r_1r_2^{2/3}r_3^{2/3}(r_3^{1/3} - r_2^{1/3})}. \]
	Recalling that both $g_{13}$ and $g_{12}$ vanish at the endpoints of $[0,\pi]$, we get the thesis. An analogous argument applies to the comparison between $g_{13}$ and $g_{23}$.
\end{proof}
\begin{rmr}
	It follows from the Lemma above that for every value of $g_{13}$, and so for every fixed $\beta$, there exists at least one $\alpha$ where $g_{12}(\alpha)$ takes the same value. If the value of $g_{13}$ is not the maximal one then there are exactly two different $\alpha$'s such that the value is achieved. The same holds for $g_{23}(\alpha-\beta)$. See figure below.
\end{rmr}

\begin{figure}
\begin{tikzpicture}
	\node[anchor=south west,inner sep=0] (image) at (0,0) {\includegraphics[width=0.9\textwidth]{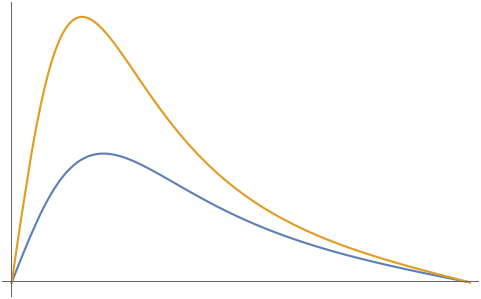}};
	\begin{scope}[x={(image.south east)},y={(image.north west)}]
	\node at (1,0) {$\beta$};
	\node at (0,1) {$\alpha$};
	\node at (0.35,0.75) {$g_{12}$};
	\node at (0.3,0.35) {$g_{13}$};
	\end{scope}
\end{tikzpicture}
\caption{The relative position of the graps of $g_{12}$ and $g_{13}$ on the interval $[0,\pi]$. However the strict inequality betwen the two maximal values is not proved. See \autoref{lemma-gij}} \label{plotg}
\end{figure}
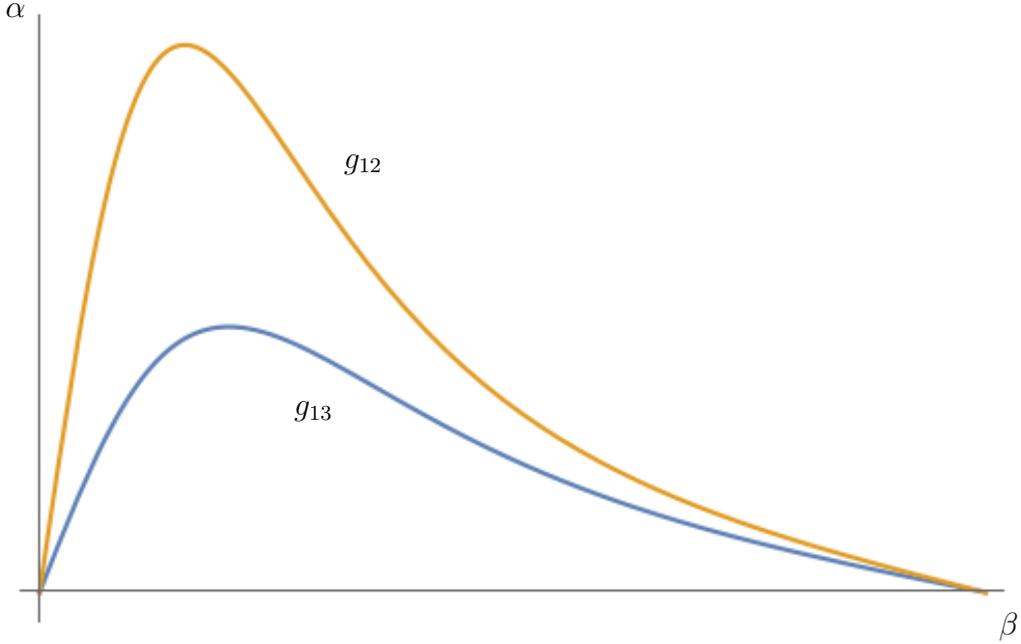

\begin{lmm} \label{lemma-h} If $\cos \theta \in (\cos \theta_{ij}, 1)$ then
	\[ g_{ij}'(\theta) < g_{ij}'(0) \frac{\cos \theta - \cos \theta_{ij}}{1 - \cos \theta_{ij}}; \]
	if $\cos \theta \in (-1, \cos \theta_{ij})$ then
	\[ g_{ij}'(\theta) < g_{ij}'(\pi) \frac{\cos\theta - \cos\theta_{ij}}{-1 - \cos \theta_{ij}}. \]
\end{lmm}
\begin{proof} We omit for simplicity of notation the indices $ij$. Recall that \[ g'(\theta) = \frac{r_ir_j Q_{ij}(\cos \theta)}{(r_i^2+r_j^2-2r_ir_j\cos \theta)^{5/2}} = h(\cos \theta), \] where $h \colon [-1,1] \to \R$, $h(t) = \frac{r_ir_j Q_{ij}(t)}{(r_i^2+r_j^2-2r_ir_jt)^{5/2}}$.
	
The thesis is a weak version of the convexity of $h$: if $h$ is convex, then the inequalities hold by applying the Jensen's inequality separately in the intervals $[-1, \cos \theta_{ij}]$ and $[\cos \theta_{ij}, 1]$. It could happen, however, that $h$ has a concave part between $-1$ and a certain threshold $\xi$, and then it is convex. In this case we prove the following:
	\begin{itemize}
		\item $h_{ij}$ is decreasing between $-1$ and a certain threshold $\sigma$, where it reaches the minimum;
		\item $\xi < \sigma$, i.e., in the interval $[\sigma, 1]$ the function is convex.
	\end{itemize}
	Then we deduce that, for $-1 \leq t \leq \sigma$,
	\[ h_{ij}(t) \leq h_{ij}(-1) \leq h_{ij}(-1) \frac{t - \cos \theta_{ij}}{-1 - \cos \theta_{ij}} \]
	(recall that $h(-1)$ is negative).
	
	On the other hand, for $\sigma \leq t \leq \cos \theta_{ij}$,
	\begin{eqnarray*}
		h(t) &\leq& \text{line joining $(\sigma, h(\sigma))$ and $(\cos \theta_{ij}, 0)$} \\
		&\leq& \text{line joining $(-1, h(-1))$ and $(\cos \theta_{ij}, 0)$}
	\end{eqnarray*}
	since $\sigma$ is a minimum point. See \autoref{ploth} for a more clear graphical meaning of the proof.
	
	\begin{figure}
		\begin{tikzpicture}
		\node[anchor=south west,inner sep=0] (image) at (0,0) {\includegraphics[width=0.9\textwidth]{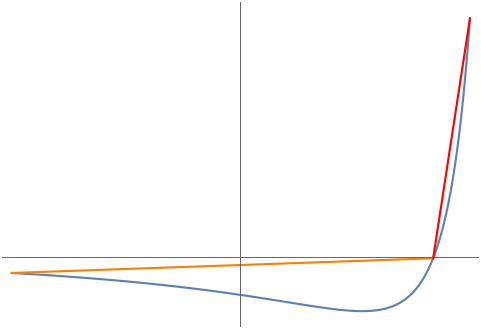}};
		\begin{scope}[x={(image.south east)},y={(image.north west)}]
		\node at (0.95,0.18) {$\cos \theta_{ij}$};
		\node at (0.4,0.08) {$h(t)$};
		\end{scope}
		\end{tikzpicture}
		\caption{A graphical understanding of \autoref{lemma-h}: the function $h(t)$ stays
			below two segments.}\label{ploth}
	\end{figure}

	Here come the computations:
	\[ h'(t) = \frac{r_i^2r_j^2t^2 + 5r_ir_j(r_i^2+r_j^2)t + r_i^4 - 13r_i^2r_j^2 + r_j^4}{(r_i^2+r_j^2-2r_ir_jt)^{7/2}}. \]
	We have that $h'(t) = 0$ for 
	\[ t = \frac{-5(r_i^2+r_j^2) \pm \sqrt{21r_i^4 + 102r_i^2r_j^2 + 21r_i^4}}{2r_ir_j}. \]
	Observe that the smaller solution is always outside the interval $[-1,1]$, since
	\[ -5(r_i^2+r_j^2) - \sqrt{21r_i^4 + 102r_i^2r_j^2 + 21r_i^4} < -\sqrt{102}r_ir_j < -2r_ir_j. \]
	Denote by $\sigma$ the bigger root.
	
	We move on to the second derivative:
	\[ h''(t) = 3r_ir_j\frac{r_i^2r_j^2t^2 + 9r_ir_j(r_i^2 + r_j^2)t + 4r_i^4 -27r_i^2r_j^2 + 4r_j^4}{(r_i^2 + r_j^2 - 2r_ir_jt)^{9/2}}. \]
	We have that $h''(t) = 0$ for
	\[ t = \frac{-9(r_i^2 + r_j^2) \pm \sqrt{65r_i^4 + 270r_i^2r_j^2 + 65r_j^4}}{2r_ir_j}. \]
	As above, the smaller root always lies outside the interval $[-1,1]$. Denote by $\xi$ the bigger root. Now we prove that $\sigma > \xi$ for any choice of the values $0 < r_i < r_j$. By homogeneity, denoting by $u = r_i^2/r_j^2 \in (0,1)$, it suffices to prove that
	\[ -5(1 + u) + \sqrt{21u^2 + 102u + 21} > -9(1+u) + \sqrt{65u^2 + 270u + 65}, \]
	\ie,
	\begin{ieee*}{C}
		4(1+u) + \sqrt{21u^2 + 102u + 21} > \sqrt{65u^2 + 270u + 65} \\
		37u^2 + 134u + 37u + 8(1+u)\sqrt{21u^2 + 102u + 21} > 65u^2 + 270u + 65 \\
		8(1+u)\sqrt{3}\sqrt{7u^2 + 34u + 7} > 28u^2 + 136u + 28 \\
		2\sqrt{3}(1+u) > \sqrt{7u^2 + 34u + 7} \\
		12(1+u)^2 > 7u^2 + 34u + 7 \\
		5(1-u)^2 > 0,
	\end{ieee*}
	as wanted.
\end{proof}

Now the idea is the following: in view of \autoref{lemma-gij}, the first equation of \eqref{stationary-system} implicitly defines two $C^\infty$ functions $\alpha_0(\beta)$ and $\alpha_\pi(\beta)$ such that $\alpha_0(0) = 0, \alpha_\pi(0) = \pi$. Analogously, the second equation implicitly defines two functions $\widehat{\alpha}_0(\beta)$ and $\widehat{\alpha}_\pi(\beta)$ such that $\widehat{\alpha}_0(0) = 0, \widehat{\alpha}_\pi(0) = \pi$.

	We want to prove that each curve $\alpha_{0,\pi}$ intersects each curve $\widehat{\alpha}_{0,\pi}$ only in 0 or $\pi$. 	
	By sign considerations, we notice that the first equation implies $\alpha(\beta) \in [\pi, 2\pi]$ and the second equation implies $\widehat{\alpha}(\beta) \in [\beta, \pi + \beta]$. Hence, the possible solutions lie in the region $\pi \leq \alpha \leq \pi + \beta$, and when considering the whole torus $\mathbb{T}^2$ the region has a ``butterfly'' shape.
	
	\begin{figure}
		\begin{tikzpicture}
		\node[anchor=south west,inner sep=0] (image) at (0,0) {\includegraphics[width=0.9\textwidth]{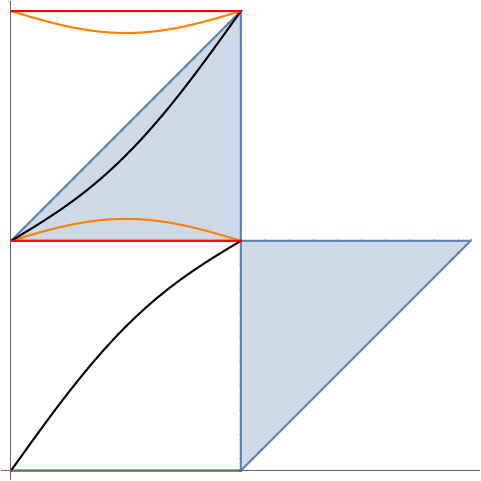}};
		\begin{scope}[x={(image.south east)},y={(image.north west)}]
		\node at (0.25,0.25) {$\hat\alpha_0$};
		\node at (0.4,0.55) {$\alpha_\pi$};
		\node at (0.4,0.75) {$\hat\alpha_\pi$};
		\node at (0.1,0.92) {$\alpha_0$};
		\node at (0,0) {$0$};
		\node at (0.5,0) {$\pi$};
		\node at (1,0) {$2\pi$};
		\node at (0,0.5) {$\pi$};
		\node at (0,1) {$2\pi$};
		\end{scope}
		\end{tikzpicture}
		\caption{In blue, the ``butterfly'' region of admissible solutions to optimality conditions \eqref{stationary-system}. In black and orange, a plot of the curves
		$\alpha_{0,\pi}$ and $\hat\alpha_{0,\pi}$ in the region $0\leq \beta \leq \pi$. } \label{butterfly}
	\end{figure}
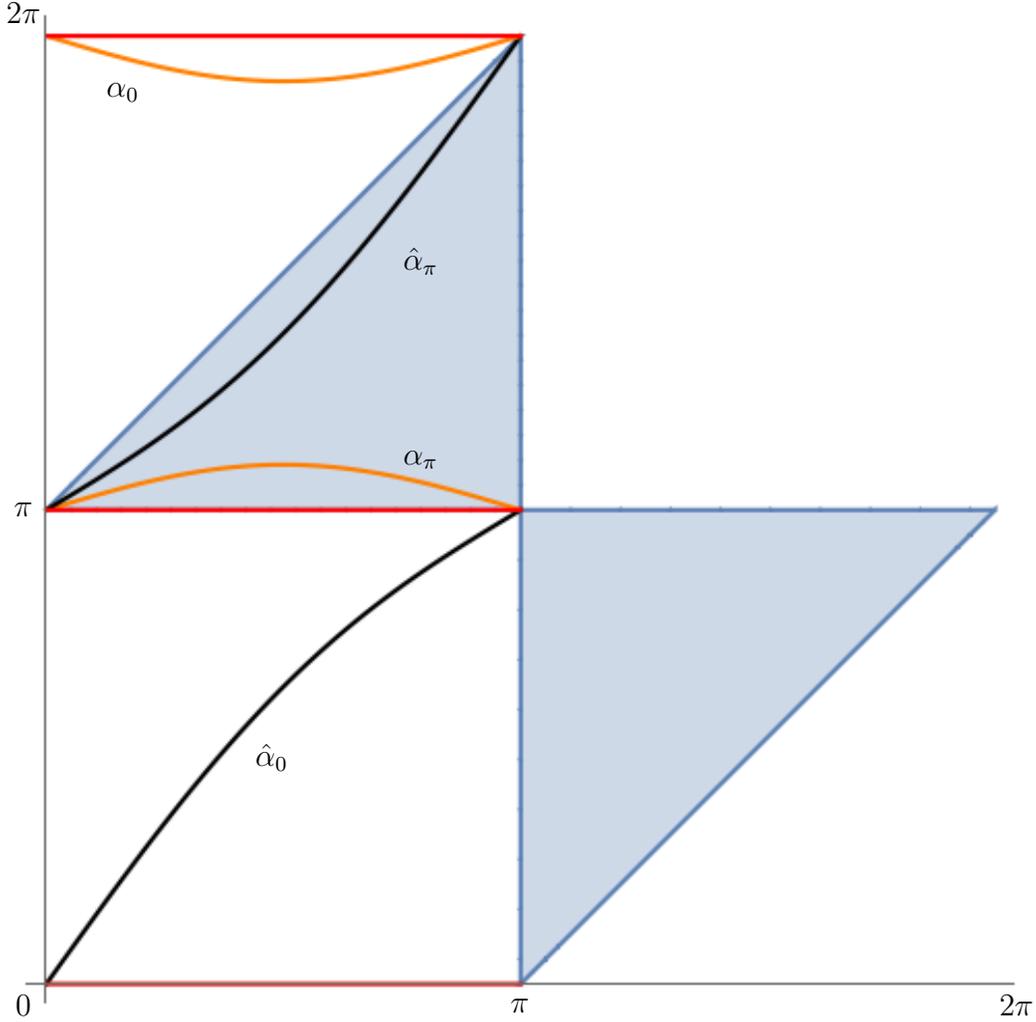
	
	This already shows that the curves $\alpha_0(\beta)$ and $\widehat{\alpha}_0(\beta)$ do not produce solutions, since we have that $\beta - \pi \leq \alpha_0(\beta) \leq 0$ and $0 \leq \widehat{\alpha}_0(\beta) \leq \pi$. Thus we can concentrate our attention on the curves $\alpha_\pi$ and $\widehat{\alpha}_\pi$. 
	
	The key observation lies in the fact that 
	\[ \pi \leq \alpha_\pi(\beta) \leq \pi + \alpha_\pi'(0)\beta, \]
	i.e., the function $\alpha_\pi(\beta)$ stays below its tangent line at $\beta = 0$ (see \autoref{tangents}). Likewise, the function $\widehat{\alpha}_\pi(\beta)$ stays above its tangent line at $\beta = 0$. This allows us to conclude that they do not intersect since, as we will see, the condition \eqref{phi1-condition} is equivalent to $\alpha_\pi'(0) \leq \widehat{\alpha}_\pi'(0)$.
	
	\begin{figure}
		\begin{tikzpicture}
		\node[anchor=south west,inner sep=0] (image) at (0,0) {\includegraphics[width=0.9\textwidth]{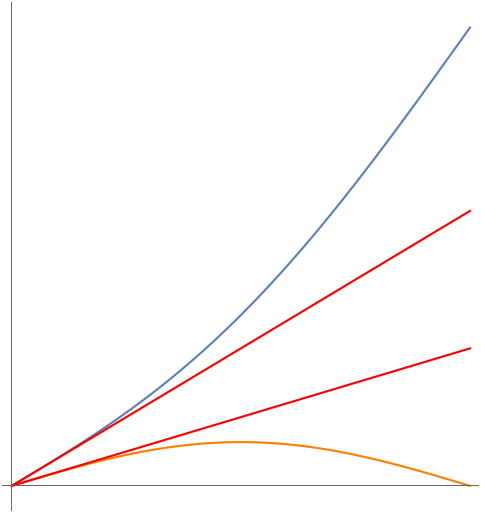}};
		\begin{scope}[x={(image.south east)},y={(image.north west)}]
		\node at (0.5,0.1) {$\alpha_\pi(\beta)$};
		\node at (0.85,0.25) {$\pi + \alpha'_\pi(0)\beta$};
		\node at (0.85,0.45) {$\pi + \hat\alpha'_\pi(0)\beta$};
		\node at (0.5,0.45) {$\hat\alpha_\pi(\beta)$};
		\node at (0.05,0) {$0$};
		\node at (0,0.1) {$\pi$};
		\node at (1,0) {$\pi$};
		\node at (0,1) {$2\pi$};
		\end{scope}
		\end{tikzpicture}
		\caption{ A graphical understanding of \autoref{lemma-alpha}: the function $\alpha_\pi(\beta)$ is confined by $\pi \leq \alpha_\pi(\beta) \leq \pi + \alpha_\pi'(0)\beta$, and similarly $\pi + \widehat{\alpha}_\pi'(0)\beta \leq \widehat{\alpha}_\pi(\beta) \leq \pi + \beta$. This implies that the intersection between $\alpha_\pi$ and $\widehat{\alpha}_\pi$ is only at $\beta = 0$.} \label{tangents}
	\end{figure}
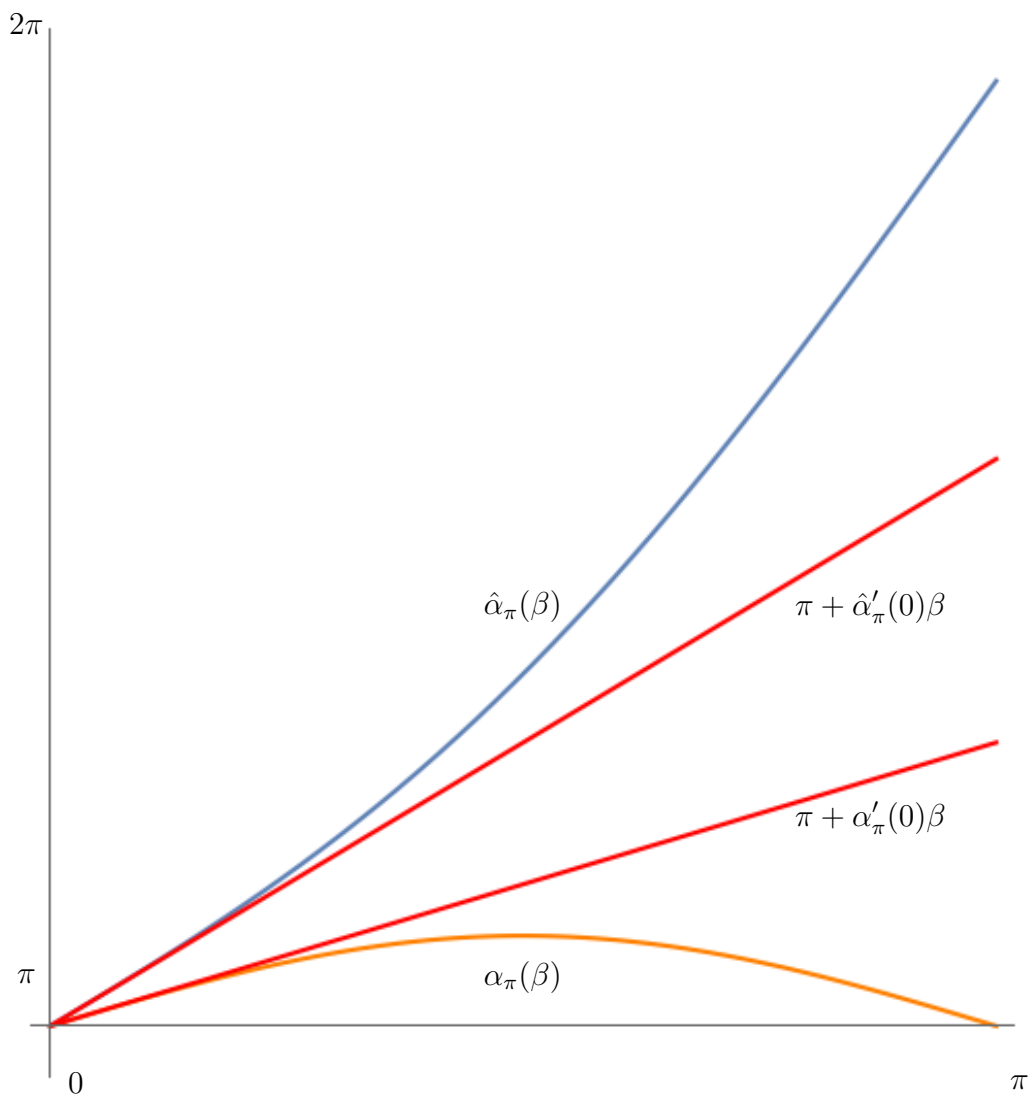
	
	\begin{lmm} \label{lemma-alpha}
		For $\beta \in (0,\pi)$ let $\alpha(\beta)$ be the solution of
		\[ \begin{cases}
		g_{13}(\beta) + g_{ij}(\alpha) = 0 \\
		\alpha(0) = \alpha(\pi) = \pi.
		\end{cases} \]
		
		Then
		\[ \pi \leq \alpha(\beta) < \pi + \alpha'(0)\beta. \]
(See \autoref{tangents} for a graphical understanding.)
	\end{lmm}
	
	\begin{proof} Differentiating in $\beta$ we get
		\[ g_{13}'(\beta) + \alpha'(\beta) g_{ij}'(\alpha(\beta)) = 0 \implies \alpha'(\beta) = \frac{g_{13}'(\beta)}{-g_{ij}'(\alpha(\beta))}. \]
		
		Take $\beta \in (0, \theta_{13})$, where $\theta_{13}$ is the critical value of $g_{13}$, so that $\cos \beta > \cos \theta_{13}$. By \autoref{lemma-gij} we have that $\alpha \in [\pi, 2\pi - \theta_{ij}]$, because the equation $g_{13}(\beta) + g_{ij}(\alpha) = 0$ has two solutions in the interval $[\pi, 2\pi]$ and by definition $\alpha$ is the leftmost one. Using \autoref{lemma-h} we have
		\[ \alpha'(\beta) \leq \frac{g_{13}'(0)}{-g_{ij}'(\pi)} \frac{\cos \beta - \cos \theta_{13}}{1 - \cos \theta_{13}} \frac{-1-\cos \theta_{ij}}{\cos \alpha(\beta) - \cos \theta_{ij}}. \]
		
		Since $\frac{g_{13}'(0)}{-g_{ij}'(\pi)} = \alpha'(0) \geq 0$, it suffices to show that
		\[ \frac{\cos \beta - \cos \theta_{13}}{1 - \cos \theta_{13}} \frac{-1-\cos \theta_{ij}}{\cos \alpha(\beta) - \cos \theta_{ij}} \leq 1. \]
		Let $\tilde{\alpha} = \alpha(\beta) - \pi$, so that $0 \leq \tilde{\alpha} \leq \beta$. We must prove
		\begin{align*}
			&\frac{\cos \beta - \cos \theta_{13}}{1 - \cos \theta_{13}} \frac{1 + \cos \theta_{ij}}{\cos \tilde{\alpha} + \cos \theta_{ij}} \leq 1. \\
			&(1 + \cos \theta_{ij}) (\cos \beta - \cos \theta_{13}) \leq (1 - \cos \theta_{13})(\cos \tilde{\alpha} + \cos \theta_{ij}) \\
			&(\cos \beta - \cos \tilde{\alpha}) + \cos \theta_{ij} \cos \beta + \cos \theta_{13} \cos \tilde{\alpha} \leq \cos \theta_{ij} + \cos \theta_{13}.
		\end{align*}
		
		But this is true, since $\tilde{\alpha} \leq \beta \implies \cos \beta - \cos \tilde{\alpha} \geq 0$ and clearly
		\[ \cos \theta_{ij} \cos \beta + \cos \theta_{13} \cos \tilde{\alpha} \leq \cos \theta_{ij} + \cos \theta_{13}. \]
		
		We got the desired inequality for $\beta \in (0, \theta_{13})$. However, for $\beta \geq \theta_{13}$ we have $\alpha'(\beta) \leq 0$, hence the line $\alpha'(0)\beta$ is increasing and the function $\alpha(\beta)$ is decreasing, giving the inequality for every $\beta$.
	\end{proof}

	By \autoref{lemma-alpha}, we obtain that the function $\alpha_\pi(\beta)$ lies between the horizontal line $\alpha = \pi$ and the line $\alpha = \pi + \alpha_\pi'(0)\beta$ (strictly for $\beta > 0$). Recall that the function $\widehat{\alpha}_\pi(\beta)$ satisfies the second equation of the stationary system (\eqref{stationary-system})
	\[ \frac{r_1r_3\sin \beta}{D_{13}(\beta)^{3/2}} - \frac{r_2r_3\sin (\widehat{\alpha} - \beta)}{D_{23}(\widehat\alpha - \beta)^{3/2}} = 0 \]
	with $\widehat{\alpha}_\pi(0) = \pi$, $\widehat{\alpha}_\pi(\pi) = 2\pi$.
	
	By a change of variables $\tilde{\alpha}(\beta) = 2\pi + \beta - \widehat{\alpha}_\pi(\beta)$, we get that $\tilde{\alpha}$ satisfies
	\[ \begin{cases}
	g_{13}(\beta) + g_{23}(\tilde{\alpha}) = 0 \\
	\tilde{\alpha}(0) = \tilde{\alpha}(\pi) = 0,
	\end{cases}
	\]
	hence $\pi < \tilde{\alpha}(\beta) < \pi + \tilde{\alpha}'(0)\beta$, i.e.,
	\[ \pi + \widehat{\alpha}_\pi'(0)\beta < \widehat{\alpha}_\pi(\beta) < \pi + \beta \]
	for $\beta > 0$.
	So the idea is that the two lines provide a separation of the curves, so that no intersection can happen except at the starting point.
	
	We conclude by observing that the condition \eqref{phi1-condition} is equivalent to $\widehat{\alpha}_\pi'(0) \geq \alpha_\pi'(0)$: indeed we have
	\begin{align*}
		&g_{13}(\beta) - g_{23}(\widehat{\alpha}_\pi(\beta) - \beta) = 0 \implies \\
		&\widehat{\alpha}_\pi'(0) = 1 +  \frac{g_{13}'(0)}{g_{23}'(\pi)} = 1 -  \frac{r_1r_3}{(r_3-r_1)^3} \frac{(r_2+r_3)^3}{r_2r_3} = \frac{r_2(r_3-r_1)^3 - r_1(r_2+r_3)^3}{r_2(r_3-r_1^3)}
	\end{align*}
	and
	\[ \alpha_\pi'(0) = \frac{g_{13}'(0)}{-g_{12}'(\pi)} = \frac{r_1r_3}{(r_3-r_1)^3} \frac{(r_1+r_2)^3}{r_1r_2} = \frac{r_3(r_1+r_2)^3}{r_2(r_3-r_1)^3}. \]
\end{proof}

\section{Consequences of \autoref{thm-phi1}}

When $\rho$ satisfies the assumptions of \autoref{thm-phi1}, we know that
\[ c(r_1, r_2, r_3) = \frac{1}{r_2+r_1} + \frac{1}{r_3+r_2} + \frac{1}{r_3-r_1} \]
for $\rho$-a.e. $(r_1, r_2, r_3) \in [0, s_1] \times [s_1, s_2] \times [s_2, +\infty)$. The key observation lies in the fact that this can be viewed as a 1-dimensional Coulomb cost for points $-r_2, r_1, r_3 \in \R$. We can now rely on a somewhat well-established theory for the Coulomb cost in dimension $d = 1$.

This allows to prove \autoref{monge-solution}.

\begin{proof}[Proof of \autoref{monge-solution}]
	This is a direct consequence of the one-dimensional result presented in \cite{colombo2013multimarginal}. Indeed, we can consider the absolutely continuous measure $\tilde{\rho} \in \mathcal{P}(\R)$ defined by
	\[ \tilde{\rho}(x) = \begin{cases} \rho(x) & x \in [0, s_1] \cup [s_2, +\infty) \\ \rho(-x) & x \in [-s_2, -s_1] \\ 0 & \text{otherwise} \end{cases} \]
	and observe that the DDI map $T$ for $\rho$ corresponds to the optimal increasing map $S$ defined in \cite{colombo2013multimarginal}.
	
	The optimality follows from the fact that $c(x, T(x), T^2(x)) = c_{1D}(y, S(y), S^2(y))$ for $\rho$-a.e. $x \in [0,s_1]$ and $\tilde{\rho}$-a.e. $y \in [-s_2, -s_1]$, where $c_{1D}$ is the Coulomb cost on the real line, as observed above.
\end{proof}

The idea of the first part of the  the proof of \autoref{thm-counterexample} is to show that on the support of the DDI map, the $c$-cyclical monotonicity is violated. We prepare a couple of technical results.

\begin{lmm} \label{M-eps-lemma1}
	Let $\frac{s_1}{s_2} > \frac{1+2\sqrt{3}}{5}$. Then there exist $\epsilon, M > 0$ such that
	\begin{equation} \label{M-eps-condition}
	\frac{2}{s_2+\epsilon} + \frac{1}{2s_2+\epsilon} + \frac{1}{2s_1+\epsilon} > \frac{\sqrt{3}}{s_1-\epsilon} + \frac{1}{s_1} + \frac{1}{M-\epsilon}.
	\end{equation}
\end{lmm}

\begin{proof}
	When $\epsilon = 0$ and $M = +\infty$, the inequality \eqref{M-eps-condition} reads
	\[ \frac{2}{s_2} + \frac{1}{2s_2} + \frac{1}{2s_1} > \frac{\sqrt{3}}{s_1} + \frac{1}{s_1}, \]
	which is equivalent to $\frac{s_1}{s_2} > \frac{2\sqrt{3}+1}{5}$.
	
	By continuity, there is a small $\epsilon$ such that
	\[ \frac{2}{s_2+\epsilon} + \frac{1}{2s_2+\epsilon} + \frac{1}{2s_1+\epsilon} > \frac{\sqrt{3}}{s_1-\epsilon} + \frac{1}{s_1}. \]
	
	Now choose $M$ big enough such that the desired inequality \eqref{M-eps-condition} holds.
\end{proof}

\begin{rmr}\label{functionfi}
We recall the polynomial condition \eqref{phi1-condition}:
\[ r_2(r_3-r_1)^3 -r_1(r_3+r_2)^3 - r_3(r_1+r_2)^3 \geq 0\,. \]
For fixed $r_1$ and $r_2$, the cubic polynomial in $r_3$ that appears on the left-hand side of \eqref{phi1-condition} has three real roots. They are given by the following expressions:
\[ -r_2, \quad \frac{5 r_1 r_2 + r_2^2 \pm (r_1+r_2) \sqrt{r_2^2 + 12r_1r_2 - 4r_1^2}}{2(r_2-r_1)}.	\]
Since we are only interested in the region where $r_3 > 0$ and since
\[ \frac{5 r_1 r_2 + r_2^2 + (r_1+r_2) \sqrt{r_2^2 + 12r_1r_2 - 4r_1^2}}{2(r_2-r_1)} \]
is the only positive root for every value of $0 < r_1 < r_2$, the condition \eqref{phi1-condition} can be rewritten as
\begin{equation} \label{phi1-root}
\varphi(r_1, r_2) := \frac{5 r_1 r_2 + r_2^2 + (r_1+r_2) \sqrt{r_2^2 + 12r_1r_2 - 4r_1^2}}{2(r_2-r_1)}\le r_3.
\end{equation}
\end{rmr}
\begin{rmr}\label{ColomboStra} 
In \cite{colombo2016counterexamples}, a crucial role was played by Lemma 4.1. In our framework this lemma can be obtained as a consequence of \autoref{thm-phi1} by choosing (following the notation of \cite{colombo2016counterexamples})  
\[r_3^->\max_{[r_1^-,r_1^+]\times [r_2^-,r_2^+]}\varphi(r_1,r_2)\,.\]
	If $r_1^+ <r_2^-$, as assumed by the authors in \cite{colombo2016counterexamples}, then the maximum above is a real number and the threshold $r_3^-$ can be fixed.
Thus our result gives a quantitative optimal version of their choice. 
Moreover, \autoref{thm-phi1} allows us to deal with the case in which there is no gap between 
$r_1^+$ and $r_2^-$, since we have an explicit control of the growth of $\varphi(r_1, r_2)$ as $r_1 \to r_2$.
\end{rmr}
Before the next lemma we introduce the notation $c_\pi(r_1,r_2,r_3)$ for the Coulomb cost of the configuration where all three points are placed along the $x$-axis so that the angle between $v_1$ and $v_2$ is $\pi$ and the angle between $v_1$ and $v_3$ is $0$; following our general line, the vector $v_1$ lies along the positive $x$-axis. In other words
\[c_\pi(r_1,r_2,r_3)=C(r_1,r_2,r_3,\pi,0)\,.\]
\begin{lmm} \label{M-eps-lemma2}
	Let $s_1,s_2,\eps$ and $M$ as in \autoref{M-eps-lemma1}, and let $(r_1,r_2,r_3) \in (0,\epsilon) \times (s_2-\epsilon,s_2) \times (s_2,s_2+\epsilon)$ and $(\ell_1,\ell_2,\ell_3) \in (s_1-\epsilon,s_1) \times (s_1,s_1+\epsilon) \times (M,+\infty)$. Suppose that the condition \eqref{phi1-root} is satisfied by both $(r_1,r_2,r_3)$ and $(\ell_1,\ell_2,\ell_3)$. Then
	\[ c(r_1,r_2,r_3) + c(\ell_1,\ell_2,\ell_3) > c(\ell_1,r_2,r_3) + c(r_1,\ell_2,\ell_3). \]
\end{lmm}

\begin{proof}
	Since the condition \eqref{phi1-root} is satisfied, we have
	\begin{align*}
		c(r_1,r_2,r_3) &= c_\pi(r_1,r_2,r_3) = \frac{1}{r_1+r_2} + \frac{1}{r_2+r_3} + \frac{1}{r_3-r_1} \\
		&\geq \frac{1}{s_2+\epsilon} + \frac{1}{2s_2+\epsilon} + \frac{1}{s_2+\epsilon}
	\end{align*} 
	and
	\begin{align*}
		c(\ell_1,\ell_2,\ell_3) &=c_\pi(\ell_1,\ell_2,\ell_3) = \frac{1}{\ell_1+\ell_2} + \frac{1}{\ell_2+\ell_3} + \frac{1}{\ell_3-\ell_1} \\
		&\geq \frac{1}{2s_1+\epsilon} + \frac{1}{\ell_2+\ell_3} + 0.
	\end{align*}
	
	Now we analyze the other side. Since $(\ell_1,\ell_2,\ell_3)$ satisfy \eqref{phi1-root} and $r_1 < \ell_1$, then also $(r_1,\ell_2,\ell_3)$ satisfy \eqref{phi1-root}\footnote{It can be computed that $\varphi(r_1,r_2)$ it increasing in $r_1$.}, so that 
	\begin{align*}
		c(r_1,\ell_2,\ell_3) &= c_\pi(r_1,\ell_2,\ell_3) = \frac{1}{r_1+\ell_2} + \frac{1}{\ell_2+\ell_3} + \frac{1}{\ell_3-r_1} \\
		&\leq \frac{1}{s_1} + \frac{1}{\ell_2+\ell_3} + \frac{1}{M-\epsilon}.
	\end{align*}
	
	For the other term we have
	\[ c(\ell_1,r_2,r_3) \leq c_{\Delta}(\ell_1,r_2,r_3) \leq c_{\Delta}(\ell_1,\ell_1,\ell_1) = \frac{\sqrt{3}}{\ell_1} \leq \frac{\sqrt{3}}{s_1-\epsilon}, \]
	where $c_{\Delta}(r_1, r_2, r_3) = C(r_1,r_2,r_3,\frac{2\pi}{3},\frac{4\pi}{3})$ denotes the cost when the angles are the ones of an equilateral triangle. The second inequality follows from the fact that we are keeping the angles fixed, but decreasing the size of the sides. By comparing the expressions and using \autoref{M-eps-lemma1} we get the desired inequality.
\end{proof}
Finally we come to the proof of \autoref{thm-counterexample}.

\textit{Proof of \autoref{thm-counterexample}}
We divide the proof in three steps: \textbf{1) }the non-optimality of the $DDI$ map, \textbf{2) }the non-optimality of the other maps in the class $\{D,I\}^3$, and \textbf{3) }the existence of smooth counterexample densities. In the first two steps we follow the ideas of \cite[Proof of Counterexample 2.7]{colombo2016counterexamples}.

\textbf{Step 1) The non-optimality of the $DDI$ map: } 
	We fix  $\eps, M$ according to \autoref{M-eps-lemma1}. Since $\rho$ is fully supported and $T$ is continuous, we have
	\[ T(x) \to s_2^- \text{ and } T^2(x) \to s_2^+ \text{ as } x \to 0, \]
	and
	\[ T(x) \to s_1^+ \text{ and } T^2(x) \to +\infty \text{ as } x \to s_1^-. \]
	
	This allows to choose triplets $(r_1,r_2,r_3)$ and $(\ell_1, \ell_2, \ell_3)$ as in the hypothesis of \autoref{M-eps-lemma2} such that 
	\[ (r_1, r_2, r_3) = (x, T(x), T^2(x)) \quad \text{and} \quad (\ell_1, \ell_2, \ell_3) = (y, T(y), T^2(y)). \]
	
	Apply \autoref{M-eps-lemma2} to conclude that the support of the DDI map is not $c$-cyclically monotone.\\\\
\textbf{Step 2) The non-optimality of the other maps in the class $\{D,I\}^3$ }
We assume that $S$ is the DID map. The proof for the other maps of the class $\{D,I\}^3$ requires slightly different choices of intervals, but the idea is the same.  It suffices to find two triples of radii $(l,S(l),S^2(l))$ and $(r,S(r),S^2(r))$ where $l,r\in[0,s_1]$ such that 
\[c(l,S(l),S^2(r))+c(r,S(r),S^2(l))<c(l,S(l),S^2(l))+c(r,S(r),S^2(r))\]
To this end, let us show that 
\begin{align}
&\text{There exist parameters }\tilde\alpha\in(0,s_1),~\tilde\beta\in(s_1,s_2)\,,\text{ and }\tilde M>s_2\text{ such that }\nonumber\\
&c=c_\pi\text{ on the set }(0,\tilde\alpha]\times[\tilde\beta,s_2)\times[\tilde M,\infty)\text{ and }\nonumber\\
&\text{for all }r\in(0,\tilde\alpha)\text{ we have }S(r)\in [\tilde\beta,s_2)\text{ and }S^2(r)\in [\tilde M,\infty)\,.\label{eq:tag1}
\end{align}
Let us pick an $\alpha\in (0,s_1)$. We set $\beta=S(\alpha)$ and $M=S^2(\alpha)$. A comment: The DID map $S$ maps the interval $(0,\alpha]$ first to the interval $[\beta,s_2)$ (by $S$) and then to the half-line $[M,\infty)$ (by $S^2$). By Lemma 4.1 in \cite{colombo2016counterexamples} we can fix a $\tilde M>s_2$ such that
\[c=c_\pi~~~\text{on the set }(0,\alpha]\times[\beta,s_2)\times[\tilde M,\infty)\,.\]
If $\tilde M\le M$, we can choose in \eqref{eq:tag1} $\tilde\alpha=\alpha$ and $\tilde\beta=\beta$, and the claim follows. If $\tilde M>M$, we choose $\tilde\alpha=S(\tilde M)$ and $\tilde\beta=S^2(\tilde M)$, because now by the monotonicities of $S$ we have that  $(0,\tilde\alpha]\subset (0,\alpha)$ and $[\tilde\beta,s_2)\subset [\beta,s_2]$. 

Now let us prove that the $DID$ map is not optimal. 
We fix intervals $I\subset[0,s_1]$ and $J\subset[s_1,s_2]$ and a halfline $H\subset[s_2,\infty)$, given by Condition \eqref{eq:tag1}. We also fix points $r,l\in I$; without loss of generality we may assume that $l<r$. By the choice of $I$, $J$ and $H$ we have $c=c_\pi$ on the Cartesian product $I\times J\times H$, and therefore we have 
\begin{equation}\label{eq:tag2}
c(r,S(r),S^2(r))=c_\pi(r,S(r),S^2(r))~~~\text{and}~~~c(l,S(l),S^2(l))=c_\pi(l,S(l),S^2(l))\,.
\end{equation}
Denoting 
\[x_1=-S(l), ~~x_2=-S(r),~~ x_3=l,~~x_4=r,~~x_5=S^2(r)\,,\text{ and }x_6 =S^2(l)\,,\]
we have six ordered points on the real line. 
As has been proved in \cite{colombo2013multimarginal}, the optimal way of coupling these points is 
\[(x_1,x_3,x_5)~~~\text{and}~~~(x_2,x_4,x_6)\,.\]
Therefore, denoting by $c_{1D}$ the one-dimensional Coulomb cost, we have 
\[c_{1D}(x_1,x_3,x_5)+c_{1D}(x_2,x_4,x_6)<c_{1D}(x_1,x_3,x_6)+c_{1D}(x_2,x_4,x_5)\,.\]
By Condition \eqref{eq:tag2}  this is equivalent to
\begin{equation}\label{eq:tag3}
c_{1D}(x_1,x_3,x_5)+c_{1D}(x_2,x_4,x_6)<c(l,S(l),S^2(l))+c(r,S(r),S^2(r))\,.
\end{equation}
By our choices the points $x_i$ we have 
\[c_{1D}(x_1,x_3,x_5)=c_\pi(l,S(l),S^2(r))~~~\text{and}~~~c_{1D}(r_2,r_4,r_6)=c_\pi(r,S(r),S^2(l))\,.\]
Combining this with Condition \eqref{eq:tag3} gives
\begin{equation}\label{eq:tag4}
c_\pi(l,S(l),S^2(r))+c_\pi(r,S(r),S^2(l))<c(l,S(l),S^2(l))+c(r,S(r),S^2(r))\,.
\end{equation}
The radial cost $c$ is obviously majorized by the cost $c_\pi$ for all radii, so we get
\begin{align*}
&c(l,S(l),S^2(r))+c(r,S(r),S^2(l))\\
&\le c_\pi(l,S(l),S^2(r))+c_\pi(r,S(r),S^2(l))\stackrel{a)}{<}c(l,S(l),S^2(l))+c(r,S(r),S^2(r))\,,
\end{align*}
where the inequality (a) is Condition \eqref{eq:tag4}. 
So all in all
\[c(l,S(l),S^2(r))+c(r,S(r),S^2(l))<c(l,S(l),S^2(l))+c(r,S(r),S^2(r))\]
contradicting the $c$-cyclical monotonicity of the map $S$ as we set out to prove. \hfill$\square$

\textbf{Step  3) The existence of smooth counterexample densities: } 
As noted in \autoref{functionfi}, the polynomial condition \eqref{phi1-condition} can be solved for $r_3$ and transformed into
\[r_3\ge \varphi (r_1,r_2)\,,\]
where 
\[ \varphi (r_1, r_2) = \frac{5 r_1 r_2 + r_2^2 + (r_1 + r_2) \sqrt{r_2^2 +
   12 r_1 r_2 - 4 r_1^2}}{2 (r_2 - r_1)}. \]
If we study this condition on the `graph' of the $DDI$-map 
\[\mathcal{G}:=\{(x,T(x),T^2(x))~|~x\in[0,s_1]\}\,,\]
that is, if we plug in $(r_1,r_2,r_3)=(x,T(x),T^2(x)))$, we see that the main assumption \eqref{eq:phi1onDDI}  of this theorem can be written in the form: 
\[
T^2(x)\ge \varphi(x,T(x))~~~\text{for }\rho\text{-a. e. }x\in (0,s_1)\,.
\]
This condition is satisfied whenever
\begin{equation}\label{eq:t2with_h}
T^2(x)=\varphi(x,T(x))+h(x)
\end{equation}
where $h:[0,s_1]\to\R$ is strictly positive with $h(0)=0$.  
 The expression \eqref{eq:t2with_h} gives us an efficient way to construct counterexample densities. We can choose the densities  $\rho_1$ and $\rho_2$ on the respective intervals $[0,s_1]$  and $[s_1,s_2]$ in any way we like. Then we define the tail for the measure (formally: $\rho|_{[s_2,\infty)}=:\rho_3$) by setting  
 \begin{equation}\label{eq:tail}
 \rho_3=(\varphi+h)_\sharp\rho_1=T^2_\sharp\rho_1
 \end{equation}
 for a suitable $\rho_1$. Above we have abbreviated $\varphi(x)=\varphi(x,T(x))$, and we will use the same abbreviation in the following. Note that once the densities $\rho_1$, and $\rho_2$ have been chosen, the function $\varphi(x)$ is well-defined for all $x\in(0,s_1)$ because the $DDI$ map $T$ that maps the first interval to the second one is determined by the densities $\rho_1$ and $\rho_2$. 
 
In the beginning of this theorem, we have already assumed that the densities $\rho|_{[0,s_2]}:=\rho_1+\rho_2$ satisfy $\frac{s_1}{s_2} > \frac{1 + 2\sqrt{3}}{5}$. If we further assume that $\rho_1, \rho_2$ are smooth and satisfies the assumption $\frac{\rho(0)}{\rho(s_2)}>\frac72$ (which will be motivated below), and then define the tail according to \eqref{eq:t2with_h} choosing a specific $h$ to be defined shortly, we actually get a smooth counterexample density. 

First we illustrate how to generate a continuous density. 
We plug in the expression of $\varphi$ the point $(x_1,x_2)=(x, T (x))$ and compute the derivative with
respect to $x$.  We omit the slightly tedious computations because they are not important for expressing the idea of the counterxample. At $x=0$ the derivative reads
\[\varphi'(0)=\tfrac12(5+T'(0)+1+T'(0)+8)=T'(0)+7\,.\]
By choosing $\rho_1$, $\rho_2$ and $h$ smooth, the continuity is clear everywhere except at the point $s_2$. In particular, we must study the condition $\rho(s_2^-)=\rho(s_2^+)$. 

Abbreviating for all $x\in (0,s_1)$  $\psi(x)=\varphi(x)+h(x)$ and using the Monge-Amp\`ere equation for $T^2 (x) = \psi (x)$ we get
\begin{equation}\label{eq:2tag1}
\rho (x) = \psi' (x) \rho (\psi (x)) \,.
\end{equation}
Analogously, using Monge-Ampere for $T (x)$ we get
\begin{equation}\label{eq:2tag2}
\rho (x) = - \rho (T (x)) T' (x)\,.
\end{equation}
Putting $x = 0$ first of the two gives us $\rho (s_2^+)$ the second one gives
us $\rho (s_2^-)$ and from the equality
\[ \psi' (0) \rho (s_2^+) = - T' (0) \rho (s_2^-) \]
we get
\[h'(0)+7+T'(0)=-T'(0)\,,\]
that is, 
\begin{equation}\label{eq:2tag3}
T'(0)=-\frac12(h'(0)+7)\,.
\end{equation}
Setting $x=0$ in \eqref{eq:2tag2} gives $T'(0)=-\frac{\rho(0)}{\rho(s_2)}$, and combining this with \eqref{eq:2tag3} leads to the condition
\[\frac{\rho(0)}{\rho(s_2)}=\frac12(h'(0)+7)\,,\]
that is,
\[h'(0)=2\frac{\rho(0)}{\rho(s_2)}-7\,.\]
This is the condition our auxiliary function $h$ must satisfy to guarantee the continuity of the density $\rho$. As will be shown below, we will also need to assume that the first derivative of $h$ is strictly positive, this is why the final assumption on the values $\rho(0)$ and $\rho(s_2)$ must read $\frac{\rho(0)}{\rho(s_2)}>\frac72$ -- that is, the inequality has to be strict. 

\textbf{About the differentiability of the counterexample densities }\\
Let us see what the condition of $\rho$ being differentiable at $s_2$ looks like, in terms of densities. The differentiability of the rest of the tail is guaranteed by the smoothness of $\rho_1$, $\rho_2$, and $\psi$. 

We write the Monge-Ampere equation for $\psi$:
\[\psi'(x)\rho(\psi(x)) = \rho(x) \quad x \in (0, s_1). \]

We differentiate $n$-times both sides with respect to $x$:
\[ \sum_{k = 0}^n \binom{n}{k} \psi^{(n-k+1)}(x) \frac{d^k}{dx^k} \rho(\psi(x)) = \rho^{(n)}(x) \]
and we compute for $x = 0$
\[ \sum_{k = 0}^n \binom{n}{k} \psi^{(n-k+1)}(0) \left[ \frac{d^k}{dx^k} \rho(\psi(x)) \right]_{x = 0} = \rho^{(n)}(0). \]

Suppose that we already defined $h'(0), \dotsc, h^{(n)}(0)$. The only term containing $h^{(n+1)}(0)$ is obtained for $k = 0$ in the LHS, and reads $\psi^{(n+1)}(0) \rho(s_2)$. Hence we can isolate it and get
\[ \psi^{(n+1)}(0) \rho(s_2) = \rho^{(n)}(0) -\sum_{k = 1}^n \binom{n}{k} \psi^{(n-k+1)}(0) \left[ \frac{d^k}{dx^k} \rho(\psi(x)) \right]_{x = 0}. \]

Since $\rho$ is positive everywhere, in particular $\rho(s_2) > 0$ and get a well-defined expression for $h^{(n+1)}(0)$ depending on $h'(0), \dotsc, h^{(n)}(0)$ (already previously defined by induction).
The base step is given by $h(0) = 0$. 

Now by Borel's lemma there exists a smooth function $f:\R\to\R$ such that $f^{(k)}=h^{(k)}$ for all natural numbers $k$. Because $h'(0)=f'(0)>0$ and $h(0)=f(0)=0$, there is a $\delta>0$ and an interval $[0,\delta]$ such that $f(x)>0$ for all $x\in[0,\delta]$ We now choose our $h$ to coincide with $f$ in the interval $[0,\tfrac{\delta}{2})$ and to be constant, equal to $f(\delta)$ on the interval $[\delta,\infty)$. On the interval $(\tfrac{\delta}{2},\delta)$ we join these two parts smoothly, so that the function $h$ is smooth on all of its domain $[0,\infty)$ -- obviously at $0$ we mean by smoothness the existence of all derivatives from the right.   
Now by using the $h$ generated above and by defining the tail density $\rho_3$ according to \eqref{eq:tail} we get the existence of smooth counterexample densities. \hfill$\square$


\section*{Acknowledgement}
The research of the author is part of the project   financed by the Italian Ministry of Research
 and is partially financed  by the {\it ``Fondi di ricerca di ateneo''}  of the  University of Firenze. 
 
The research of the author is part of the project  {\it "Analisi puntuale ed asintotica di energie di tipo non locale collegate a modelli della fisica"}  of the  GNAMPA-INDAM.

The research of the third author is part of the project \emph{Contemporary topics on multi-marginal optimal mass transportation}, funded by the Finnish Postdoctoral Pool (Suomen Kulttuuris\"a\"ati\"o). 


\bibliographystyle{plain}

\end{document}